\documentclass[11pt]{amsart}

\usepackage[T1]{fontenc}
\usepackage{lmodern}
\usepackage{microtype}
\usepackage[a4paper,margin=28mm]{geometry}
\usepackage{amsmath,amssymb,amsthm,mathtools}
\usepackage{enumitem}
\usepackage{mathrsfs}
\usepackage[hidelinks]{hyperref}
\usepackage{todonotes}
\usepackage{comment}

\newtheorem{theorem}{Theorem}[section]
\newtheorem{proposition}[theorem]{Proposition}
\newtheorem{lemma}[theorem]{Lemma}

\theoremstyle{definition}

\theoremstyle{remark}
\newtheorem{remark}[theorem]{Remark}

\numberwithin{equation}{section}

\newcommand{\C}{\mathbb C}
\newcommand{\E}{\mathbb E}
\newcommand{\F}{\mathscr F}

\renewcommand{\P}{\mathbb P}
\newcommand{\cP}{\mathcal P}
\newcommand{\Q}{\mathbb Q}
\newcommand{\R}{\mathbb R}

\newcommand{\1}{\mathbf 1}
\newcommand{\cM}{\mathsf{M}}
\newcommand{\cD}{\mathsf D}

\newcommand{\norm}[1]{\lVert #1\rVert}
\newcommand{\abs}[1]{\lvert #1\rvert}

\newcommand{\spn}{\operatorname{span}}
\newcommand{\dd}{\,\mathrm d}

\newcommand{\Sigmaa}[1]{\Sigma_{#1}}
\newcommand{\one}{\1}

\newcommand{\wt}{\widetilde}

\title[Continuity of Stochastic Convolutions]{No Free Lunch for Continuity of Stochastic Convolutions}

\author{Mark Veraar}
\address[Mark Veraar]{Delft Institute of Applied Mathematics\\
Delft University of Technology \\ P.O. Box 5031\\ 2600 GA Delft\\The
Netherlands.} 
\email{m.c.veraar@tudelft.nl} 

\author{Joris van Winden}
\address[Joris van Winden]{Mathematical Institute, Leiden University\\ Einsteinweg 55\\ 2333 CC Leiden \\ The Netherlands.}
\email{j.van.winden@math.leidenuniv.nl} 

\thanks{The first author has received funding from the VICI subsidy VI.C.212.027 of the Netherlands Organisation for Scientific Research (NWO). The second author has received funding from the VICI subsidy VI.C.242.087 of the Netherlands Organisation for Scientific Research (NWO)}

\subjclass[2020]{Primary 60H15, 47D06; Secondary 60H05, 60G17, 47A60.}

\keywords{Stochastic convolution, maximal inequality, path continuity, analytic semigroup, sectorial operator, square-function estimate, $H^\infty$-calculus.}

\begin{document}

\begin{abstract}
We construct compact, exponentially stable $C_0$-semigroups $S$
on Hilbert spaces $X$ such that, for every $T>0$, the stochastic
convolution
\[
    U_g(t)=\int_0^t S(t-s)g(s)\,\dd\beta_s
\]
is unbounded on $[0,T]$ with positive probability for some predictable
$g\in L^\infty(\Omega;L^2(0,T;X))$, where $\beta$ is a real Brownian
motion. One example is even an analytic semigroup. For the other, the negative generator
$A$ has sectorial angle $\pi/2$ and a bounded
$H^\infty$-calculus.

Our main tool is a necessary condition: for exponentially stable
semigroups, an $L^2$-maximal estimate on $\mathbb R_+$ forces a lower
square-function estimate for the negative generator.
We deduce this implication from a new identity involving the
stochastic convolution, the subordinated Poisson semigroup, and
a stopped Brownian motion.
Combining this condition with Schauder multipliers on a conditional
trigonometric basis yields counterexamples to the maximal estimate.
An extrapolation argument then produces the integrands with
unbounded stochastic convolutions.

Finally, an energy-adapted chaining argument gives continuity and
maximal estimates for arbitrary $C_0$-semigroups under the stronger
condition $g\in L^2([0,T]; L^\infty(\Omega;X))$.
\end{abstract}

\maketitle

\section{Introduction}
\label{sec:introduction}
Let \(-A\) generate a $C_0$-semigroup
\(S=(S(t))_{t\geq 0}\) on a separable Hilbert space \(X\).  Let  
\((\Omega,\F,\P)\) be a complete probability space with a filtration $(\F_t)_{t\geq 0}$ carrying a real
Brownian motion \((\beta_t)_{t \geq 0}\).  For a predictable process \(g : \R_+ \times \Omega \to X\) which is square
integrable in time,  the stochastic convolution
\begin{equation}
 U_g(t):=\int_0^t S(t-s)g(s)\,\dd\beta_s
 \qquad t\geq0,
 \label{eq:mild-solution}
\end{equation}
is the mild solution of
\begin{equation}
 \dd U_g(t)+AU_g(t)\,\dd t=g(t)\,\dd\beta_t,
 \label{eq:stochastic-equation}
\end{equation}
with zero initial value.  Although the driving martingale
\(M(t)=\int_0^t g(s)\,\dd\beta_s\) has continuous paths, \(U_g\) is in
general not a martingale, and the Burkholder--Davis--Gundy inequality does
not directly control its pathwise supremum.

The long-standing endpoint problem addressed in this paper is therefore:
\begin{itemize}
    \item {\em Does \(U_g\) have a continuous version when
\(g\) is predictable and square integrable,
for an arbitrary \(C_0\)-semigroup \(S\)?}
\end{itemize}
Its quantitative form asks whether, for \(T>0\) and \(0<p<\infty\), the following \emph{maximal estimate} holds:
\begin{equation}
 \E\sup_{0\leq t\leq T}\norm{U_g(t)}^p
 \leq C_{p,T,S}^p
 \E\norm{g}_{L^2([0,T];X)}^p.
 \label{eq:intro-maximal-inequality}
\end{equation}
Kotelenez proved an energy estimate for contraction semigroups in
1982~\cite{Kotelenez1982}, but the corresponding assertion for an arbitrary
\(C_0\)-semigroup remained unresolved.  Da Prato and Zabczyk explicitly
recorded the \(L^2\)-in-time problem as open in
\cite[p.~144]{DPZ92}.  Hausenblas and Seidler later stated the
path-continuity question~\cite[p.~108]{HausenblasSeidler2008}.  In 2020, van Neerven
and Veraar isolated the precise formulation used here as Question~6.1 and
emphasised that it was open even for Hilbert spaces
\cite{vanNeervenVeraar2020}.  Thus, from Kotelenez's 1982 paper until 2026,
the endpoint question remained open for at least forty-four
years.

In this article we resolve it in the negative. We construct two different examples in which the stochastic convolution is unbounded with positive probability, even for one-dimensional Brownian noise and integrands $g$ with uniformly bounded total energy $\|g\|_{L^2(0,T;X)}\leq C$ a.s. Both semigroups are compact and exponentially stable; one is analytic, and the for the other the leading operator $A$ admits a bounded $H^\infty$-functional calculus (which is known to imply \eqref{eq:intro-maximal-inequality} in the case of analytic semigroups). We also prove a complementary positive result: continuity and maximal estimates hold for arbitrary $C_0$-semigroups whenever the integrand is non-random.

\subsection{Overview of sharp maximal inequalities}
For the sharp \(L^2\)-in-time maximal estimate, three proof mechanisms have emerged:
\begin{enumerate}[label=\textup{(\arabic*)},leftmargin=*,itemsep=1pt,topsep=3pt]
\item time discretisation;
\item It\^o's formula;
\item group dilation theory.
\end{enumerate}

Kotelenez's original Hilbert space argument belongs to the first category
and requires the semigroup to be contractive~\cite{Kotelenez1982}.  A later
approach, close in spirit to discretisation, yields \(L^p\)-maximal
estimates in \(2\)-smooth spaces with the sharp \(O(\sqrt p)\) dependence
as \(p\to\infty\)~\cite{vanNeervenVeraar2022}.  It uses martingale
inequalities of Pinelis~\cite{Pinelis}, which Seidler had earlier used to
obtain the corresponding \(O(\sqrt p)\) Burkholder--Davis--Gundy bound for
ordinary stochastic integrals in \(2\)-smooth spaces~\cite{Sei10}.  Surprisingly, the
discretisation argument is sufficiently flexible to allow the generator,
or more generally the evolution family, to depend adaptively on
\((t,\omega)\).

The second approach applies It\^o's formula to a suitable function of the
solution, e.g. $\|x\|^p$.  In Hilbert spaces this was developed by Tubaro~\cite{Tubaro1984}
and Ichikawa~\cite{Ichikawa1986}.  Brze\'zniak and Peszat gave an
extension to a subclass of the  \(2\)-smooth Banach spaces~\cite{BrzezniakPeszat2000}, and
van Neerven and Zhu subsequently obtained the result for all \(2\)-smooth
spaces~\cite{vanNeervenZhu2011} by an extension of the It\^o formula to all such spaces. A direct application of It\^o's formula
to \(x\mapsto\norm{x}^p\) typically gives constants of order \(p\).  The
sharp order \(\sqrt p\) can nevertheless be recovered by first proving an exponential
tail estimate, using It\^o's formula for a regularisation such as
\(x\mapsto(1+\lambda\norm{x}^2)^{1/2}\), see
\cite{BrzezniakPeszat2000,vanNeervenVeraar2022}. In \cite{coxvw2024sharp} it was shown that these tail bounds can be used to reproduce the $\sqrt{p}$ behavior for $p\to \infty$. 

Both the time discretisation and It\^o's formula argument require a (quasi-)contractive
semigroup or evolution family. A group dilation argument has the advantage that the original semigroup need not itself be contractive: once the convolution is represented by a group on a larger space, the
maximal estimate follows directly from the 
Burkholder--Davis--Gundy inequality, and consequently gives the optimal order \(\sqrt p\) and often even the best possible constant.  This method was used by Hausenblas and Seidler for contraction semigroups on
Hilbert spaces and for positive contraction semigroups on \(L^q\)-spaces
\cite{HausenblasSeidler2001,HausenblasSeidler2008}.  In the analytic semigroup setting, the $H^\infty$-functional-calculus of angle $<\pi/2$ can be used to construct group 
dilations on \(2\)-smooth spaces; the results of Seidler~\cite{Sei10} and
Veraar and Weis~\cite{VeraarWeis2011} provided the first non-contractive
instances of the sharp endpoint estimate. 

If one is willing to strengthen the time integrability, continuity is
available for arbitrary \(C_0\)-semigroups.  The factorisation method of Da
Prato, Kwapie\'n, and Zabczyk~\cite{DaPratoKwapienZabczyk1987, DPZ92} gives the
result when \(g(\cdot,\omega)\in L^p(0,T;X)\) almost surely for some
\(p>2\); subsequent steps
include \cite{Seidler1993,Brzezniak1997,Ver10JEE}, where space-time regularity was obtained in the parabolic setting. The strict inequality \(p>2\)
is intrinsic to this argument.  In the parabolic setting the
factorisation method can also be replaced by fractional Sobolev estimates.
This has the additional advantage that the operator may depend on
\((t,\omega)\) in an adapted way; see \cite{PronkVer}.

\begin{remark}
    The positive results recalled above are often formulated for an \(H\)-cylindrical Wiener process, with integrands taking values in \(\mathscr{L}_2(H,X)\), the space of Hilbert--Schmidt operators.
Our counterexample is formulated for the special case of a real Brownian motion, corresponding to \(H=\mathbb R\) and the canonical identification \(\mathscr{L}_2(H,X) =  X\).
\end{remark}

\subsection{Main results}
The above results have required substantial work and exploit additional structure of the semigroup. Nevertheless, it remained conceivable that this structure was needed only by the available proof methods, and that strong continuity alone would suffice for path continuity at the sharp endpoint. Our counterexamples show that there is no such free lunch.

Our first main result shows that the maximal estimate and path continuity can fail for a certain analytic semigroup.

\begin{theorem}[Example of an analytic semigroup]
\label{thm:main}
There exist a separable Hilbert space \(X\) and a
densely defined operator \(A\) on \(X\) with the following properties.
\begin{enumerate}[label=\textup{(\roman*)}]
\item\label{it:main1} \(-A\) generates a compact and exponentially stable analytic semigroup
\(S\); the operator \(A\) is sectorial of angle zero.
\item\label{it:main2} There exists a stochastic basis and a 
predictable process $g\in L^\infty\bigl(\Omega;L^2([0,1];X)\bigr)$
such that $U_g$ as defined in \eqref{eq:mild-solution}
satisfies
\begin{equation*}
 \P\Big(
  \sup_{t\in[0,1]\cap\Q}\norm{U_g(t)}=\infty
 \Big)>0.
\end{equation*}
In particular, \(U_g\) has no version with almost surely bounded or continuous paths and \eqref{eq:intro-maximal-inequality} does not hold.
\end{enumerate}
\end{theorem}

As mentioned before, for analytic semigroups, boundedness of the $H^\infty$-functional calculus implies the maximal estimate  (see \cite{Sei10} and \cite{VeraarWeis2011}), which we will prove in an alternative way in Proposition \ref{prop:Maxsquare}. 
This raises the question whether the condition that $S$ is analytic can be omitted.
The next result shows that this is not the case.

\begin{theorem}[Example with a bounded $H^\infty$-calculus]
\label{thm:main2}
There exist a separable Hilbert space \(X\) and a
densely defined operator \(A\) on \(X\) with the following properties.
\begin{enumerate}[label=\textup{(\roman*)}]
\item\label{it1:main2} \(-A\) generates a compact and exponentially stable $C_0$-semigroup
\(S\); the operator \(A\) is sectorial of angle $\pi/2$.
\item\label{it2:main2} $A$ has a bounded $H^\infty$-calculus with $\omega_{H^\infty}(A) = \pi/2$. 
\item\label{it3:main2} There exists a stochastic basis and a predictable process $g\in L^\infty\bigl(\Omega;L^2([0,1];X)\bigr)$
such that $U_g$ as defined in \eqref{eq:mild-solution}
satisfies
\begin{equation*}
 \P\Big(
  \sup_{t\in[0,1]\cap\Q}\norm{U_g(t)}=\infty
 \Big)>0.
\end{equation*}
In particular, \(U_g\) has no version with almost surely bounded or continuous paths and \eqref{eq:intro-maximal-inequality} does not hold.
\end{enumerate}
\end{theorem}

Note that the driving martingales in the examples above are nevertheless continuous and have
uniformly bounded pathwise quadratic variation.  The obstruction is
therefore created by the interaction with the semigroup, rather than by
cylindrical noise as is the case in \cite{IscoeEtAl1990}.

Our final main result is a complementary positive endpoint.  For
an arbitrary \(C_0\)-semigroup, Theorem~\ref{thm:chaining} gives continuity
and a maximal estimate, with \(O(\sqrt p)\) moment growth, under the stronger
deterministic-envelope assumption
\[
 \int_0^T\norm{g(t,\cdot)}_{L^\infty(\Omega;X)}^2\,\dd t<\infty.
\]
This displays the distinction between
\(L^\infty(\Omega;L^2([0,T];X))\) and deterministic pointwise control in
\(L^2([0,T])\). In particular, for any nonrandom $g\in L^2(0,T;X)$, $U_g$ has continuous paths. This implies that randomness of the processes $g$ in Theorems \ref{thm:main} and \ref{thm:main2} is necessary. 

\subsection{Proof strategy and organisation}
The proof of Theorems~\ref{thm:main} and~\ref{thm:main2} consists of three main steps.
The first (and perhaps the most significant) step is to prove a new connection between square function estimates and maximal estimates. This is done in Section~\ref{sec:squarefunc}.
The main result there is Theorem~\ref{thm:necessary-square}, which shows that a maximal estimate for $S$ implies a certain lower square function estimate for the negative generator of $S$.

The second step is to construct an operator which has the good properties listed in Theorem~\ref{thm:main}, but fails the lower square function estimate.
This is done in Section~\ref{sec:schauder} using Schauder multipliers on a conditional Schauder basis consisting of trigonometric polynomials.
After that in Section \ref{sec:endpoint-calculus}, we construct a variant of the previous example which has the good properties listed in  Theorem \ref{thm:main2}, but still fails the maximal estimate \eqref{eq:intro-maximal-inequality}. 

The third step is to turn failure of the maximal estimate into failure of path continuity for predictable integrands in the space $L^\infty\bigl(\Omega;L^2([0,1];X)\bigr)$.
This is done in Section~\ref{sec:extrapolation} using known extrapolation arguments.

Finally, in Section~\ref{sec:special-processes} we prove Theorem~\ref{thm:chaining}, which states the positive result for a smaller class of integrands.

\subsection{Notation and preliminaries}
Throughout the paper, $-A$ generates a $C_0$-semigroup $(S(t))_{t \geq 0}$ (often denoted just by $S$) on a complex separable Hilbert space $X$.

\subsubsection{Probability}
A stochastic basis is a complete filtered probability space $(\Omega, \F, \P, (\F_t)_{t \geq 0})$ which carries a (real, one-dimensional) Brownian motion $(\beta_t)_{t \geq 0}$. Let $I = [0,T]$ or $I=\R_+$. 

For $p,q \in (0,\infty]$, we use the notation $L^p_{\cP}(\Omega;L^q(I;X))$ for the closed subspace of $L^p(\Omega;L^q(I;X))$ consisting of predictable processes.

As mentioned before, for a predictable process $g : I \times \Omega \to X$ which is $\P$-almost surely in $L^2(I;X)$, we write $U_g(t)$ for the process
\begin{equation}
    U_g(t) \coloneq \int_0^t S(t-s)g(s) \dd \beta_s, \qquad t \in I,
\end{equation}
using a continuous version whenever one is available.
It should be clear from context which semigroup is used at any point.

For $p \in (0,\infty)$, we write $\cM_{p,I}(S) \in [0,\infty]$ for the least admissible constant for which the inequality
\begin{equation}
 \left(\E\sup_{t\in I}\norm{U_g(t)}^p\right)^{1/p}
 \leq \cM_{p,I}(S)
 \norm{g}_{L^p(\Omega;L^2(I;X))}
 \label{eq:maximal-inequality}
\end{equation}
holds for all predictable step processes $g$.
Note that $\cM_{p,I}(S)$ may depend on the stochastic basis.
We say that $S$ satisfies a maximal estimate (with $p$ and $I$) if $\cM_{p,I}(S) < \infty$.
In this case, a standard approximation argument shows that $U_g$ has a continuous version for all $g \in L^p_{\cP}(\Omega;L^2(I;X))$ and~\eqref{eq:maximal-inequality} extends to all such $g$. In Remark \ref{rem:extrapolationMp} we show that $\cM_{p_0,I}(S) < \infty$ implies that $\cM_{p,I}(S) < \infty$ for all $p>0$.

In the special case where $S$ is the identity (thus $A=0$), the estimate \eqref{eq:maximal-inequality} reduces to 
\begin{align}
    \left(\E\sup_{t\in I}\norm{\int_0^t g(s) \dd \beta_s}^p\right)^{1/p} \leq B_p \norm{g}_{L^p(\Omega;L^2(I;X))},
 \label{eq:BDG}
\end{align}
which holds for any $p\in (0,\infty)$, as follows from the well-known Burkholder--Davis--Gundy inequalities. 
One can even show that $B_p\leq \frac{2p^{3/2}}{p-1}$ for $p>1$. In particular, $B_p \leq 4\sqrt{p}$ for $p\geq 2$. This is well-known and can be deduced from \cite[Theorem A]{CarlenKree} and \cite[Lemma A.2]{VeraarWeis2011}.

\subsubsection{Functional calculus}
Suppose that $A$ is an injective sectorial operator. We write $\omega(A)$ for the sectorial angle of $A$. For $\theta \in (0,\pi)$ we write $\Sigmaa{\theta}=\{z\in\C\setminus\{0\}:|\arg z|<\theta\}$, and let $H_0^\infty(\Sigmaa{\theta})$ consist of holomorphic functions on $\Sigma_{\theta}$ satisfying $|f(z)|\lesssim\min\{|z|^\varepsilon,|z|^{-\varepsilon}\}$
for some $\varepsilon>0$.

Let $\omega(A) < \theta < \pi$.
For $f \in H^{\infty}_0(\Sigma_{\theta})$, the usual Cauchy integral against
the resolvent defines the operator $f(A)\in\mathscr{L}(X)$. We say that $A$ has a
bounded $H^\infty(\Sigmaa{\theta})$-calculus if there is a constant $C_{\theta}$ such that  
\begin{align}\label{eq:Hinfty}
\norm{f(A)}\leq C_\theta\norm{f}_\infty \ \ \text{for all
$f\in H_0^\infty(\Sigmaa{\theta})$}. 
\end{align}
For nonzero $f \in H_0^\infty(\Sigmaa{\theta})$, define
\begin{equation}
 \norm{x}_{f,A}
 :=\left(\int_0^\infty\norm{f(tA)x}^2\,\frac{\dd t}{t}\right)^{1/2},
 \qquad x\in X,
 \label{eq:square-function}
\end{equation}
allowing the value $\infty$. McIntosh's theorem
\cite{McIntosh1986} (see also
\cite[Theorem~7.3.1]{Haase2006} and
\cite[Theorem~10.4.21]{HNVW2017}) states that $A$ has a bounded
$H^\infty(\Sigmaa{\theta})$-calculus if and only if
$\norm{x}_{f,A}\eqsim \norm{x}$ for all $x\in X$. Moreover, in that case 
\[\omega_{H^\infty}(A) := \inf\{\theta>\omega(A): \text{\eqref{eq:Hinfty} holds}\} = \omega(A).\]

For any nonzero $f_1, f_2\in
H_0^\infty(\Sigma_{\theta})$, McIntosh--Yagi
\cite[Theorem~5]{McIntoshYagi} gives
\[
 \norm{x}_{f_1,A}\eqsim \norm{x}_{f_2,A},
 \qquad x\in X,
\]
(with the implicit constants depending only on $f_1, f_2, A$) without assuming a bounded $H^\infty$-calculus for $A$.

\section{Square functions and maximal estimates}
\label{sec:squarefunc}

In this section, we relate square-function estimates to maximal estimates
for stochastic convolutions. 
We will use square functions for the cases
\begin{equation}
 \psi(z):=z^{1/2}e^{-z},\qquad
 \phi(z):=z^{1/2}e^{-z^{1/2}},
 \label{eq:defpsiphi}
\end{equation}
with the principal square root. 
Observe that
$\psi\in H_0^\infty(\Sigma_{\theta})$ for $0<\theta<\pi/2$, and
$\phi\in H_0^\infty(\Sigma_{\theta})$ for $0<\theta<\pi$.
In particular, if $\omega(A)<\pi/2$, then 
\begin{align}\label{eq:psiphi}
\norm{x}_{\psi,A}\eqsim \norm{x}_{\phi,A}, \ \ x\in X,
\end{align}
where the implicit constants only depend on $A$. Note that $\omega(A)<\pi/2$ is equivalent to $-A$ generating a bounded analytic semigroup (see \cite[Theorem G.5.2]{HNVW2017}).

To highlight the connection between square functions and maximal estimates, we begin by proving the following proposition.
The assumption that $\omega(A)<\pi/2$ is essential as we will see in Theorem \ref{thm:endpoint-counterexample}. 

\begin{proposition}[Square-function estimates imply maximal estimate]\label{prop:Maxsquare}
    Suppose that $\omega(A) < \pi / 2$ and that there exist $0 < c$ and $C < \infty$ such that we have the following two-sided square function estimate:
\begin{equation}
    \label{eq:sqrexp}
    c \norm{x}
    \leq \norm{x}_{\psi,A} \leq C \norm{x}, \quad x \in X.
\end{equation}
Then for any $p \in (0,\infty)$, $S$ satisfies a maximal estimate \eqref{eq:maximal-inequality} on $\R_+$ with constant
\begin{equation*}
    \cM_{p,\R_+}(S) \leq B_p\, c^{-1}C,
\end{equation*}
where $B_p$ is as in~\eqref{eq:BDG}.
\end{proposition}
Since the two-sided square-function estimate~\eqref{eq:sqrexp} is equivalent to boundedness of the $H^\infty$ calculus, Proposition \ref{prop:Maxsquare} can also be deduced from \cite[Theorem 2.2]{Sei10} or \cite[Theorem 1.1]{VeraarWeis2011} mentioned in the introduction. Below we give a more direct argument by `embedding' $U_g(t)$ inside a martingale which takes values in the enlarged space $L^2(\R_+;X)$.

\begin{proof}
    Let $g$ be a predictable step process, and set $\mathscr{X} = L^2(\R_+;X)$.
    By expanding the definition of $\norm{x}_{\psi,A}$ and using $e^{-tA} = S(t)$, we rewrite~\eqref{eq:sqrexp} into a more transparent form:
    \begin{align}
        \label{eq:sqrexpexplicit}
        c^2 \norm{x}^2 \leq \int_0^{\infty}\norm{A^{1/2}S(r)x}^2 \dd r \leq C^2 \norm{x}^2, \quad x \in X.
    \end{align}
    We now define the $\mathscr{X}$-valued predictable process
    \begin{align*}
        \Phi(s) \coloneq r \mapsto \one_{s < r}A^{1/2}S(r-s)g(s),
    \end{align*}
    as well as the $\mathscr{X}$-valued martingale
    \begin{align*}
        \mathcal{M}(t) \coloneq \int_0^t \Phi(s)\dd \beta_s,
    \end{align*}
    and observe that for $r > t$, the semigroup identity yields
    \begin{equation*}
        \mathcal{M}(t)(r) = \int_0^t 1_{s < r}A^{1/2}S(r-s)g(s) \dd \beta_s = A^{1/2}S(r-t)U_g(t).
    \end{equation*}
    From the lower estimate of~\eqref{eq:sqrexpexplicit}, we thus find for $t \in \R_+$
    \begin{align*}
        c^2 \norm{U_g(t)}^2 &\leq \int_0^{\infty} \norm{A^{1/2} S(r)U_g(t)}^2 \dd r
        = \int_t^{\infty} \norm{A^{1/2} S(r-t)U_g(t)}^2 \dd r 
        \\&\quad= \int_t^{\infty}\norm{\mathcal{M}(t)(r)}^2 \dd r \leq \norm{\mathcal{M}(t)}_{\mathscr{X}}^2.
    \end{align*}
    On the other hand, from the upper estimate of~\eqref{eq:sqrexpexplicit}, we find for $s \in \R_+$
    \begin{align*}
        \norm{\Phi(s)}_{\mathscr{X}}^2 = \int_s^{\infty}\norm{ A^{1/2}S(r-s)g(s)}^2 \dd r
        = \int_0^{\infty}\norm{ A^{1/2}S(r)g(s)}^2 \dd r \leq C^2 \norm{g(s)}^2.
    \end{align*}
    Thus, combining the two estimates above and using \eqref{eq:BDG} in between, we conclude that
    \begin{align*}
        \norm{U_g}_{L^p(\Omega;C(\R_+;X))} 
        &\leq c^{-1}
        \norm{\mathcal{M}}_{L^p(\Omega;C(\R_+;\mathscr{X}))} 
        \\ & \leq B_p c^{-1}
        \norm{\Phi}_{L^p(\Omega;L^2(\R_+;\mathscr{X}))} 
        \\ & \leq B_p c^{-1}C\norm{g}_{L^p(\Omega;L^2(\R_+;X))}. \qedhere
    \end{align*}
    
\end{proof}
The main result of this section provides a partial converse to the previous proposition, but with $\phi$ instead of $\psi$.
By contraposition, it will allow us to construct a semigroup $S$ for which the maximal estimate fails, forming the basis for our counterexample.
\begin{theorem}[Maximal estimate implies lower square function estimate]
\label{thm:necessary-square}
Suppose that $S$ is exponentially stable and that $\cM_{2,\R_+}(S) < \infty$.
Then the following lower square function estimate holds:
\begin{equation}
 \norm{x}\leq 2\, \cM_{2,\R_+}(S) \norm{x}_{\phi,A}
 \qquad x\in X.
 \label{eq:necessary-lower-square}
\end{equation}
\end{theorem}

Before we commence the proof, we introduce several auxiliary objects which play a key role.
Since $S$ is exponentially stable, $A$ is injective and sectorial with $\omega(A) \leq \pi / 2$.
Hence, we may define $A^{1/2}$ via sectorial functional calculus as described in the preliminaries, and set
\[
 C\coloneq (2A)^{1/2}, \qquad P_r\coloneq e^{-rC},\quad r\geq 0.
\]
Thus \((P_r)_{r \geq 0}\) is the subordinated Poisson semigroup, and it inherits the exponential stability of $S$.
Secondly, for $a > 0$, we define $b^a_t \coloneq a + \beta_t$ and define the (a.s. finite) stopping time
\[
 \tau^a=\inf\{t\geq0:b^a_t = 0\}.
\]
The main trick in the proof is to apply It\^o's formula to the process $t \mapsto \langle U_g(t), P^*_{b^a_t}y\rangle$ and evaluate its expectation at time $\tau^a$ using optional stopping.
Exploiting a cancellation of the drift then results in the remarkable identity~\eqref{eq:stoppingidentity} ahead.
From this identity and the maximal estimate,~\eqref{eq:necessary-lower-square} follows by a duality argument.

We first record an occupation-time identity for the killed Brownian motion $b^a_t$.
\begin{lemma}[Green's formula for killed Brownian motion]
\label{lem:green-killed-bm}
Let $f : \R_+ \to \R$ be nonnegative and measurable. Then we have the identity
\begin{equation}
\label{eq:green-killed-bm}
    \E\int_0^{\tau^a}f(b^a_s)\dd s
    =
    2\int_0^\infty (a\wedge r)f(r)\dd r.
\end{equation}

\end{lemma}

\begin{proof}
It suffices to prove \eqref{eq:green-killed-bm} for nonnegative
continuous compactly supported functions $f:\R\to[0,\infty)$.
Indeed, evaluating the two sides at indicators defines two Borel
measures on $\R$. Once the identity is established for such test
functions,  \cite[Theorem~2.25]{Kallenberg2021} implies that the two measures coincide. The identity for arbitrary nonnegative Borel measurable
functions follows by approximation with simple functions and
monotone convergence.

For $f$ as above, we define $F(a)$ as the right-hand side of~\eqref{eq:green-killed-bm}.
Then, $F$ is bounded with two bounded derivatives, and satisfies
\[
    F(0)=0,
    \qquad
    F'(s)=2\int_s^\infty f(r)\dd r,
    \qquad
    F''(s)=-2f(s), \qquad s \geq 0.
\]
Applying It\^o's formula to $F(b^a_{t \wedge \tau^a})$ then gives
\begin{equation*}
F(b^a_{t \wedge \tau^a}) = F(a) + \int_0^{t \wedge \tau^a} F'(b^a_s) \dd \beta_s - \int_0^{t \wedge \tau^a}f(b^a_s) \dd s, \qquad t \geq 0,
\end{equation*}
so that taking expectations and rearranging the terms results in
\begin{equation*}
    \E \, F(b^a_{t \wedge \tau^a}) + \E \int_0^{t \wedge \tau^a} f(b^a_s) \dd s = F(a) , \qquad t \geq 0.
\end{equation*}
Taking $t \to \infty$ using dominated convergence, the first term on the left-hand side vanishes and we recover the desired identity.
\end{proof}

We additionally need the following two lemmas, which provide useful identities for the processes $P_{b^a_t}x$ and $P_{b^a_t}^*y$.
\begin{lemma}
    \label{lem:integrable}
Let $a>0$ and $x\in \cD(A)$. Then 
    \begin{equation}
        \label{eq:kernelidentity}
        \E \int_0^{\tau^a} \norm{C P_{b^a_s}x}^2 \dd s = 2 \int_0^{\infty} (a \wedge r) \norm{C P_{r}x}^2 \dd r < \infty.
    \end{equation}
    Moreover, we have the identity
    \begin{equation}
        \label{eq:limitidentity}
        \lim_{a\to \infty}\E \int_0^{\tau^a} \norm{C P_{b^a_s} x}^2 \dd s = 
        2 \int_0^{\infty} r \norm{C P_{r}x}^2 \dd r 
        = \norm{x}_{\phi,A}^2.
    \end{equation}
\end{lemma}
\begin{remark}
    \label{rem:integrable}
    The lemma also holds true with $A$, $C$, $P$ replaced by $A^*$, $C^*$, $P^*$.
\end{remark}
\begin{proof}
    The first identity of~\eqref{eq:kernelidentity} is a direct application of Lemma~\ref{lem:green-killed-bm}.
    The finiteness follows since $x \in \cD(A)$ and $P_r$ is exponentially stable in the operator norm. 
The first identity of~\eqref{eq:limitidentity} follows by taking $a \to \infty$ in~\eqref{eq:kernelidentity} using monotone convergence, and the second identity holds by expanding the definition of $C$ and $P_{r}$ and making the substitution $\sqrt{t} = \sqrt{2}r$, resulting in
\begin{equation*}
    2 \int_0^{\infty} r \norm{(2A)^{1/2} e^{-r(2A)^{1/2}}x}^2 \dd r = \int_0^{\infty} \norm{(tA)^{1/2}e^{-(tA)^{1/2}}x}^2 \frac{\dd t}{t} = \norm{x}_{\phi,A}^2. \qedhere
\end{equation*}
\end{proof}

\begin{lemma}
    For $y \in \cD(A^*)$, it holds that
    \begin{equation}
        \label{eq:itoPbt}
        \dd P_{b^a_t}^*y = A^* P_{b^a_t}^* y \dd t - C^* P^*_{b^a_t}y \dd \beta_t
    \end{equation}
    in the strong sense, up to time $\tau^a$.
\end{lemma}
\begin{proof}
    Apply It\^o's formula to $f(b^a_{t \wedge \tau^a})$ with $f(s) = \langle x, P^*_{s} y\rangle = \langle x, e^{-s C^*}y\rangle$ for $x \in X$. 
\end{proof}
The core of the proof is the following identity, which is obtained by an optional stopping argument and the aforementioned cancellation of the drift.

\begin{lemma}
    \label{lem:stoppingidentity}
Suppose that $\cM_{2,\R_+}(S) < \infty$. 
    For every $g \in L^2_{\cP}(\Omega;L^2(\R_+;X))$ and $y \in \cD(A^*)$ we have the identity
    \begin{equation}
        \label{eq:stoppingidentity}
        \E \langle U_g (\tau^a), y\rangle = -\E\int_0^{\tau^a} \langle g(s), C^* P_{b^a_s}^*y\rangle \dd s.
    \end{equation}
\end{lemma}
\begin{remark}
    The expectation on the right-hand side is well-defined by Cauchy--Schwarz, \eqref{eq:kernelidentity}, and Remark~\ref{rem:integrable}.
\end{remark}

\begin{proof}
    By density and the maximal estimate it suffices to prove the identity for $g \in L^2_{\cP}(\Omega;L^2(\R_+;\cD(A)))$.
    In this case, $U_g$ has a continuous version which satisfies
    \begin{equation*}
        \dd U_g(t) =-A U_g(t) \dd t + g \dd\beta_t
    \end{equation*}
    in the strong sense.
    Combining this with~\eqref{eq:itoPbt}, it follows by applying It\^o's formula to $f(U_g(t \wedge \tau^a), P^*_{b^a_{t \wedge \tau^a}}y)$ with $f(u,v) = \langle u, v\rangle$  that
    \begin{equation*}
    \begin{aligned}
        \dd \langle U_g(t), P_{b^a_t}^*y\rangle = -&\langle A U_g, P_{b^a_t}^* y\rangle \dd t + \langle g, P_{b^a_t}^* y\rangle \dd \beta_t \\
        +& \langle U_g, A^* P^*_{b^a_t}y\rangle \dd t - \langle U_g, C^* P^*_{b^a_t}y\rangle \dd \beta_t \\
        -&\langle g, C^* P^*_{b^a_t}y \rangle \dd t,
    \end{aligned}
    \end{equation*}
    up to time $\tau^a$.
    Observing the cancellation between the drifts, we conclude from the identity above that
    \begin{equation*}
        \langle U_g(t \wedge \tau^a), P_{b^a_{t\wedge \tau^a}}^*y\rangle + \int_0^{t \wedge \tau^a} \langle g(s), C^* P^*_{b^a_s}y \rangle \dd s = 
        \int_0^{t \wedge \tau^a} \langle g(s), P_{b^a_s}^* y\rangle - \langle U_g(s), C^* P^*_{b^a_s}y\rangle \dd \beta_s.
    \end{equation*}
    From the maximal estimate~\eqref{eq:maximal-inequality} and Remark~\ref{rem:integrable} applied to the left-hand side of this identity, it follows that the martingale on the right-hand side is uniformly integrable. Thus, we may take the expectation, use optional stopping, and let $t\to \infty$ to conclude~\eqref{eq:stoppingidentity}.
\end{proof}
We now deduce the square function estimate~\eqref{eq:necessary-lower-square} from the maximal estimate~\eqref{eq:maximal-inequality} using~\eqref{eq:stoppingidentity} and a duality argument.
This is the only point where we will use the maximal estimate in a quantitative rather than qualitative way.
\begin{proof}[Proof of Theorem~\ref{thm:necessary-square}]
    Fix $a > 0$, $y \in \cD(A^*)$ and $g \in L^2_{\cP}(\Omega;L^2(\R_+;X))$.
    Taking the squared modulus of~\eqref{eq:stoppingidentity}, and applying Jensen's inequality and the maximal estimate to the left-hand side, we find:
    \begin{equation*}
        \abs{\E \int_0^{\tau^a} \langle g(s),  C^* P_{b^a_s}^*y \rangle \dd s}^2 \leq \cM_{2,\R_+}^2(S) \norm{y}^2 \E \int_0^{\tau^a} \norm{g(s)}^2 \dd s.
    \end{equation*}
    By duality in $L^2_{\cP}(\R_+ \times \Omega;X)$, we then obtain
    \begin{equation*}
        \E \int_0^{\tau^a} \norm{C^* P_{b^a_s}^* y}^2 \dd s \leq \cM_{2,\R_+}^2(S) \norm{y}^2.
    \end{equation*}
    Taking $a \to \infty$ using~\eqref{eq:limitidentity} then gives the adjoint upper square function estimate:
    \begin{equation*}
        \norm{y}_{\phi,A^*}^2 \leq \cM_{2,\R_+}^2(S) \norm{y}^2.
    \end{equation*}
    Since $y \in \cD(A^*)$ was arbitrary, the above extends by density to all $y \in X^*$.
    The desired lower square function estimate~\eqref{eq:necessary-lower-square} then follows from a standard duality argument for square functions (see \cite[p. 196]{Haase2006}). The constant is $2$ due to $\int_0^\infty \phi(t)^2 \frac{\dd t}{t} =\frac12$. 
\end{proof}

\begin{remark}[Maximal estimates for \(A\) and \(A^*\)]
\label{rem:maximal-both-Hinfty}
In case that $\cM_{2,\R_+}(S) < \infty$ and $\cM_{2,\R_+}(S^*) < \infty$ both hold true, the proof of Theorem~\ref{thm:necessary-square} and the fact that $A^{**} = A$ give that $\norm{x}_{\phi,A} \eqsim \|x\|$. By the previously mentioned theorem of McIntosh this implies the boundedness of the $H^\infty$-calculus. In the case $\omega(A)<\pi/2$, it follows from \eqref{eq:psiphi} that \eqref{eq:sqrexp} holds as well. 
\end{remark}

\section{Schauder basis}
\label{sec:schauder}
In this section we exhibit an operator which fails a square function estimate of the form~\eqref{eq:necessary-lower-square}.
This is done by standard arguments involving conditional bases on Hilbert spaces. Similar constructions can be found in \cite{BaillonClement1991} (see \cite{Fackler2015,Haase2006,HNVW2017} for historical details).
The reader is referred to \cite[Definition 1.1.1]{AlbiacKalton-Book} and \cite{Singer1970} for details.

Throughout this section, we fix $0 < \alpha < 1$ and define the weight $w(\theta) = \abs{\theta}^{-\alpha}$.
    Let $(f_n)_{n \geq 1}$ be the trigonometric system on the torus $\mathbb{T} = (-1/2, 1/2]$ with the symmetric ordering $0, -1, 1, -2, 2$, and so on.
    We define $X \coloneq L^2(\mathbb{T},w)$, and simply write $\norm{\cdot}$ for the corresponding norm.

We observe the following classical fact, see e.g. \cite[Example 10.2.32]{HNVW2017} or \cite[Proposition~2.3]{Nielsen2009}.
\begin{lemma}
    The sequence $(f_n)_{n \geq 1}$ is a Schauder basis of $X$.
\end{lemma}
We now additionally fix $\mu > 1$ and introduce the unbounded operator $C$ which acts as the 
Schauder multiplier $C f_k = \mu^k f_k$, with its domain $\cD(C)$ being the closure of the span of $\{f_n\}_{n \geq 1}$ with respect to the graph norm, and set $A \coloneq C^2$.

Recalling the definitions~\eqref{eq:square-function} and~\eqref{eq:defpsiphi}, the main result of this section reads as follows.
\begin{proposition}
    \label{prop:sqrfuncexample}
    The operator $A$ is sectorial with $\omega(A) = 0$, and $-A$ generates a compact exponentially stable analytic semigroup.
   The upper estimate $\norm{x}_{\phi,A} \lesssim \norm{x}$ holds uniformly in $x\in X$, but the lower estimate $\norm{x} \lesssim \norm{x}_{\phi,A}$  does not.
\end{proposition}

The failure of the lower square function estimate will be a consequence of the following lemma.
\begin{lemma}
    For all finitely supported $(a_k)_{k \geq 1}$, it holds that
    \begin{equation}
    \norm{\sum_k a_k f_k}_{\phi,A}^2 
    \lesssim \norm{(a_k)_{k \geq 1}}_{\ell^2}^2 \label{eq:QC-upper-coefficient-estimate}.
\end{equation}
\end{lemma}
\begin{proof}
A direct integration yields
\[
  \norm{\sum_k a_kf_k}_{\phi,A}^2
 =
 \sum_{j,k\geq 1}a_j\overline{a_k}B_{jk},
 \qquad
 B_{jk}
 :=
 \frac{2\mu^{j+k}}{(\mu^j+\mu^k)^2}\langle f_j,f_k\rangle.
\]
Since $\norm{f_n} \lesssim 1$, we can estimate the matrix coefficients as $\abs{B_{jk}} \lesssim \mu^{-\abs{j - k}}$ uniformly in $j$ and $k$. As a consequence, we find the following bounds for the row and column sums:
\begin{equation}
 \sup_{j\geq 1}\sum_{k\geq 1}|B_{jk}| + \sup_{k\geq 1}\sum_{j\geq 1}|B_{jk}| \lesssim 1.
\end{equation}
By the (unweighted) Schur test, it follows that $\|B\|_{\mathcal L(\ell^2)} \lesssim 1$, which yields the desired estimate.
\end{proof}

\begin{proof}[Proof of Proposition~\ref{prop:sqrfuncexample}.]
    The claims in the first sentence follow directly from \cite[Lemmas~9.1.1 and 9.1.2]{Haase2006} and \cite[Proposition 10.2.28]{HNVW2017}.
    To see the failure of the lower square-function estimate, let $n \geq 1$, define $a_{n,k} \coloneq \one_{k \leq 2n+1}$ for $k \geq 1$, and set
    \begin{equation}
        x_n(\theta) \coloneq \sum_k a_{n,k} f_k(\theta) = \sum_{\abs{\ell} \leq n} e^{2 \pi i \ell \theta} = \frac{\sin\big((2n+1)\pi \theta \big)}
     {\sin(\pi \theta)}, \quad \theta \in \mathbb{T}.
    \end{equation}
    For a sufficiently small absolute constant $c>0$,
    it holds that $|x_n(\theta)|\gtrsim n$ whenever $0<\theta <c/n$.
    Consequently,
    \begin{align}\label{eq:normxn}
     \|x_n\|^2
     =
     \int_{\mathbb T}|x_n(\theta) |^2w(\theta)\,\dd \theta
     \gtrsim
     n^2\int_0^{c/n}\theta^{-\alpha}\,\dd\theta
     \gtrsim n^{1+\alpha}, \quad n \geq 1.
    \end{align}
    On the other hand, for the coefficients we have $\norm{(a_{n,k})_{k \geq 1}}_{\ell^2}^2 \eqsim n$.
    Thus the estimate $\norm{x_n} \lesssim \norm{(a_{n,k})_{k \geq 1}}_{\ell^2}$ fails, which implies via~\eqref{eq:QC-upper-coefficient-estimate} that $\norm{x_n} \lesssim \norm{x_n}_{\phi,A}$ must also fail.

    For the upper estimate, let $x = \sum_k a_k f_k$ with $(a_k)_{k\geq 1}$ finitely supported and note that
\[\norm{x}_{\phi,A}^2 
    \lesssim \norm{(a_k)_{k \geq 1}}_{\ell^2}^2 = \norm{x}_{L^2(\mathbb T)}^2 \leq \|x\|^2,\]
    where  we used \eqref{eq:QC-upper-coefficient-estimate}, Parseval's identity and $w(\theta)\geq 1$. The general case follows by density. 
\end{proof}

\section{A counterexample with upper and lower square function bounds}
\label{sec:endpoint-calculus}

Theorem \ref{thm:necessary-square} shows the necessity of the lower square function bound for any exponentially stable $C_0$-semigroup (and thus by translation for any $C_0$-semigroup). On the other hand, in Proposition \ref{prop:Maxsquare} we showed that for exponentially stable {\em analytic} semigroups the lower and upper square function bounds imply the sharp maximal estimates. In this section, we show that in the latter result one cannot omit the analyticity. 
More precisely, we prove the following:
\begin{theorem}
\label{thm:endpoint-counterexample}
There exist a Hilbert space $X$ and an operator $A$ on
$X$ such that
\begin{enumerate}[label=\textup{(\roman*)}]
\item $-A$ generates a compact and exponentially stable $C_0$-semigroup $S$;
\item\label{it:endpointexample} $\omega(A)=\pi/2$, and $A$ satisfies the two-sided
      square-function estimate
      \begin{equation}
          \label{eq:endpointsqrfunc}
           \norm{x}
           \lesssim \norm{x}_{\phi,A}
           \lesssim \norm{x},
           \qquad x\in X;
        \end{equation}
\item The maximal inequality \eqref{eq:maximal-inequality} with $p = 2$ and $I = \R_+$ fails, i.e., $\cM_{2,\R_+}(S) = \infty$.
\end{enumerate}
\end{theorem}
Recall that $\phi$ was defined in \eqref{eq:defpsiphi}.

As previously mentioned, McIntosh's theorem shows that \ref{it:endpointexample} is equivalent to $A$ having a bounded $H^{\infty}$-calculus, which in this case has angle $\omega_{H^\infty}(A)=\pi/2$.
The key observation needed in the proof is the fact that $\cM_{p,I}$ is invariant under purely imaginary shifts of the generator. 
\begin{proposition}\label{prop:unimodular}
Let $S$ be a $C_0$-semigroup, let $\eta \in \R$ and define $S_{\eta}(t) = e^{-i\eta t}S(t)$.
Then $\cM_{p,I}(S) = \cM_{p,I}(S_{\eta})$ for any $p \in (0,\infty)$ and any interval $I$.
\end{proposition}
\begin{proof}
    This follows by writing 
    \begin{align*}
       \int_0^t S_\eta(t-s)g(s)\,\dd\beta_s =  e^{-it\eta} \int_0^t S(t-s) e^{is\eta}g(s)\,\dd\beta_s 
    \end{align*}
    and using that the unimodular factors do not change the norms.
\end{proof}

We now add large imaginary shifts to finite-dimensional restrictions of the
Schauder multiplier from Section~\ref{sec:schauder}. 
This results in a sequence of operators which satisfy two-sided square-function estimates uniformly, while their maximal constants tend to infinity.
Taking their orthogonal sum gives the desired example to prove Theorem~\ref{thm:endpoint-counterexample}.

For our construction, we let $A$, $(f_k)_{k \geq 1}$, and $X$ be as in Section~\ref{sec:schauder} (with $\alpha, \mu$ fixed once and for all). We define
\[
 X_n \coloneq \spn\{f_1,\ldots,f_{2n+1}\}, \qquad A_n \coloneq A|_{X_n},
\]
as well as
\[
B_n \coloneq n(A_n+ K \norm{A_n} i I), \qquad S_n(t) \coloneq e^{-tB_n},
\]
with $K > 0$ to be chosen shortly.
We claim the following.

\begin{proposition}[Finite-dimensional estimates]
\label{prop:endpoint-blocks}
There exists $\gamma > 0$ such that the following estimates hold uniformly in $n$:

\begin{align}
    \label{eq:Snexpstab}
 \norm{S_n(t)}&\lesssim e^{-\gamma n t}, \qquad t\ge0, \\
    \label{eq:SnMblowup}
  \cM_{2,\R_+}(S_n) &\gtrsim n^{\alpha/2}, \\
  \label{eq:squarefunctionbound}
\norm{x} &\lesssim \norm{x}_{\phi,B_n} \lesssim \norm{x}, \qquad x\in X_n.
\end{align}
\end{proposition}

\begin{proof}
For the stability bound on $S_n$, we first observe that $S_n(t / n) = e^{-t K \norm{A_n} i} e^{-t A_n}$.
Since $A_n$ is a restriction of $A$, \eqref{eq:Snexpstab} can then be deduced from the exponential stability of $e^{-tA}$ stated in~Proposition~\ref{prop:sqrfuncexample}.
Next we prove the lower bound on $\cM_{2,\R_+}(S_n)$.
Letting $x_n$ be as in the proof of Proposition~\ref{prop:sqrfuncexample}, we obtain from \eqref{eq:QC-upper-coefficient-estimate} and \eqref{eq:normxn} (note that $x_n \in X_n$):
\begin{equation*}
  \norm{x_n}_{\phi,nA_n}^2 
  = \norm{x_n}_{\phi,A_n}^2
  =\norm{x_n}_{\phi,A}^2\lesssim n, \qquad  \norm{x_n}^2\gtrsim n^{1+\alpha}.
\end{equation*}
Therefore, Proposition~\ref{prop:unimodular} with $\eta = -n K \norm{A_n}$ and Theorem~\ref{thm:necessary-square} yield
\begin{equation*}
 \cM_{2,\R_+}(S_n) = \cM_{2,\R_+}((e^{-t nA_n})_{t \geq 0})
 \ge\frac{\norm{x_n}}{2\norm{x_n}_{\phi,nA_n}}
 \gtrsim n^{\alpha/2}.
\end{equation*}

It remains to check the square-function estimates \eqref{eq:squarefunctionbound}.
By invariance of the square function, it actually suffices to prove~\eqref{eq:squarefunctionbound}
with $B_n$ replaced by the operator
\begin{equation}
    \wt{B}_n \coloneq (n K\norm{A_n})^{-1} B_n = iI + (K\norm{A_n})^{-1}A_n.
\end{equation}
By treating $\wt{B}_n$ as a perturbation of $iI$ and choosing $K$ sufficiently large (independent of $n$), it follows from~\cite[Theorem 16.2.8]{HNVW23}
that $\wt{B}_n$ has an $H^{\infty}(\Sigma_{\omega})$-calculus for $\omega \in (\pi/2,\pi)$, with a constant uniform in $n$.
By McIntosh's theorem, this yields~\eqref{eq:squarefunctionbound} uniformly in $n$.

\end{proof}

\begin{proof}[Proof of Theorem \ref{thm:endpoint-counterexample}]
Take the blocks from Proposition~\ref{prop:endpoint-blocks} and put
\[
 \widetilde{X}=\left(\bigoplus_{n\ge1}X_n\right)_{\ell^2},\quad
 \widetilde{A} =\bigoplus_{n\ge1} B_n, \quad 
 \cD(\widetilde{A})=\Big\{(u_n)\in \widetilde{X}:
                  \sum_{n\ge1}\norm{ B_nu_n}^2<\infty\Big\},
\]
so that $-\widetilde{A}$ generates the product semigroup
\[
 \widetilde{S}(t)=\bigoplus_{n\ge1}S_n(t).
\]
Exponential stability of $\widetilde{S}$ follows from~\eqref{eq:Snexpstab}.
Since the block norms tend to zero
for each $t>0$ (again by~\eqref{eq:Snexpstab}), $\widetilde{S}(t)$ is compact, and thus $\widetilde{A}^{-1}$ as well.  

Boundedness of $\widetilde{S}$ gives $\omega(\widetilde{A})\le\pi/2$. 
On the other hand,
the eigenvalues $n(\mu^2+i K \norm{A_n})$ of $\widetilde{A}$ have arguments tending to $\pi/2$ and thus $\omega(\widetilde{A})=\pi/2$.

For the square function estimates, observe that the functional calculus of $\widetilde{A}$ acts blockwise. Hence,
\[
 \norm{x}_{\phi,\widetilde{A}}^2 
 =\sum_{n\ge1}\norm{x_n}_{\phi,B_n}^2, \qquad x \in \widetilde{X},
\]
so that~\eqref{eq:endpointsqrfunc} follows using~\eqref{eq:squarefunctionbound}. 

Finally, testing \eqref{eq:maximal-inequality} on a single block gives $\cM_{2,\R_+}(\widetilde{S}) \ge\sup_{n\ge1}\cM_{2,\R_+}(S_n)
 =\infty$.
\end{proof}

\section{Extrapolation}
\label{sec:extrapolation}
In this section, $S$ denotes a general $C_0$-semigroup on a Hilbert space $X$.  For $I = [0,T]$ or $I = \R_+$ let us introduce the notation
\begin{equation}
    \mathcal E_\infty(I)
 :=L^\infty_{\cP}(\Omega;L^2(I;X)).
\end{equation}

In this section, we will prove the main results Theorems \ref{thm:main} and \ref{thm:main2}. Combining the results of Sections \ref{sec:squarefunc}-\ref{sec:endpoint-calculus}, we already obtain suitable semigroups $S$ such that $\cM_{2,\R_+}(S) = \infty$ (i.e., the failure of the maximal estimate with $p = 2$ and $I = \R_+$) on every stochastic basis. 
Below we will strengthen this to assertions on the unboundedness of the paths of the solution $U_g$ on $[0,T]$ with positive probability even for integrands $g \in \mathcal{E}_{\infty}([0,T])$. For this we will use an extrapolation argument which goes back to Burkholder \cite{Burkholder1973}. To keep the argument self-contained and focused on the counterexamples, we allow the stochastic basis to vary; see Remark~\ref{rem:fixed-basis} for the corresponding statements on a fixed stochastic basis.

Our proof will consist of several steps, suggested by the following proposition. 
\begin{proposition}[Qualitative to quantitative]
    \label{prop:extrap}
    Let $I=[0,T]$ or $I = \R_+$ and assume that $\sup_{t \in I}\norm{S(t)}_{\mathscr{L}(X)}<\infty$.
    Then the following are equivalent.
    \begin{enumerate}[label=(\arabic*)]
        \item\label{it1:extrap} For every stochastic basis and every \(g\in\mathcal E_\infty(I)\), it holds almost surely that
        \[
            \sup_{t \in I \cap \Q} \norm{U_g(t)} < \infty.
        \]
        \item\label{it2:extrap} For every $\varepsilon > 0$ and every stochastic basis there exists $K_{\varepsilon} < \infty$ such that for every \(g\in\mathcal E_\infty(I)\) we have the estimate
         \begin{equation*}
          \P\Big(
          \sup_{t \in I \cap \Q} \norm{U_g(t)} >K_\varepsilon
           \norm{g}_{\mathcal E_\infty(I)}
          \Big)\leq\varepsilon.
         \end{equation*}
        \item\label{it3:extrap} For every $\varepsilon > 0$ there exists $K_{\varepsilon} < \infty$ such that for every stochastic basis and every  \(g\in\mathcal E_\infty(I)\) we have the estimate
         \begin{equation*}
          \P\Big(
          \sup_{t \in I \cap \Q} \norm{U_g(t)} >K_\varepsilon
           \norm{g}_{\mathcal E_\infty(I)}
          \Big)\leq\varepsilon.
         \end{equation*}
        \item\label{it5:extrap} For every $0 < p < \infty$  it holds that $\cM_{p,I}(S) < \infty$, where the constant is uniform over all stochastic bases.
    \end{enumerate}
\end{proposition}

\begin{remark}
    It is an immediate consequence of Proposition~\ref{prop:extrap} that finiteness of $\cM_{p,I}(S)$ (uniform in the stochastic basis) is independent of $0 < p < \infty$.
\end{remark}

We will prove \ref{it1:extrap}$\to$\ref{it2:extrap}$\to$\ref{it3:extrap}$\to$\ref{it5:extrap}, with the implication \ref{it5:extrap}$\to$\ref{it1:extrap} being obvious.

\begin{proof}[Proof of \ref{it1:extrap} implies \ref{it2:extrap}]
Fix $\varepsilon > 0$, and for $m \geq 1$ define
\[
 F_m=\left\{g\in\mathcal E_\infty(I):
  \P\Bigl(\sup_{t \in I \cap \Q} \norm{U_g(t)}>m \Bigr)\leq\frac\varepsilon2\right\}.
\]
Using the assumption on $U_g$, it follows that $\bigcup_{m \geq 1} F_m = \mathcal E_\infty(I)$.
We also claim that each \(F_m\) is closed.  Indeed, suppose that
\(g_n\in F_m\) and \(g_n\to g\) in \(\mathcal E_\infty(I)\). Letting $(D_N)_{N\geq 1}$ be a sequence of finite sets which increase to $I \cap \Q$, we find for every $\eta > 0$ that
\[
 \begin{split}
 \P\Big(
  \max_{t\in D_N}\norm{U_g(t)}>m+\eta
 \Big)
 &\leq
 \P(\sup_{t \in I \cap \Q} \norm{U_{g_n}(t)}>m)
 +\P\Big(
   \max_{t\in D_N}\norm{U_{g - g_n}(t)}>\eta
  \Big) \\
  &\leq
 \frac{\varepsilon}{2} +\P\Big(
   \max_{t\in D_N}\norm{U_{g - g_n}(t)}>\eta
  \Big).
 \end{split}
\]
Taking $n \to \infty$ and using that $N$ is finite, we see that the left-hand side is bounded by $\varepsilon / 2$.
We then take $N \to \infty$ and $\eta \downarrow 0$ (in this order) to conclude that $g \in F_m$.
The Baire category theorem now gives \(m\), \(g_0\), and \(r>0\) such that
\(B(g_0,r)\subseteq F_m\).  If
\(\norm{h}_{\mathcal E_\infty(I)}<r\), then \(g_0\) and \(g_0+h\) both
belong to \(F_m\), and linearity and the union bound give
\[
 \P(\sup_{t \in I \cap \Q} \norm{U_h(t)} > 2m)\leq\varepsilon.
\]
By homogeneity, we subsequently recover the stated bound in \ref{it2:extrap} for all $g \in \mathcal E_\infty(I)$.
\end{proof}

\begin{proof}[Proof of \ref{it2:extrap} implies \ref{it3:extrap}]
We use the strategy of showing that a product probability space inherits the worst constant $K_{\varepsilon}$ of its factors, roughly speaking.

Suppose \ref{it3:extrap} does not hold. Put $a_n=2^{-n/2}$.
Then there exist an $\varepsilon_0>0$, stochastic bases $(\Omega^n, \F^n,(\F_t^n)_{t\geq 0}, \P^n)$  with Brownian motions
$\beta^n$, and predictable processes $g_n:I\times \Omega^n\to X$ satisfying
$\norm{g_n}_{L^\infty(\Omega^n;L^2(I;X))}\le 1$ and
\[
 \P^n\Big(\sup_{t\in I\cap\Q}\norm{U^{\beta^n}_{g_n}(t)}
       >n/a_n\Big)>\varepsilon_0,
\]
where the superscript $\beta^n$ indicates which Brownian motion is used in the stochastic convolution.
We write $(\Omega, \F,(\F_t)_{t\geq 0}, \P)$ for the completed product probability space, which we equip with the Brownian motion $\beta=\sum_{n\ge1}a_n\beta^n$.
We identify random variables on $\Omega^n$ with their lifts to $\Omega$, and additionally write $\P^{\not n}$ for the product probability measure of $(\P^j)_{j\geq 1,j \neq n}$.

By definition of $\beta$, we can decompose
\[
 U^\beta_{g_n}=a_nU^{\beta^n}_{g_n}+V_n,
 \qquad
 V_n(t)=\int_0^t S(t-s)g_n(s)\,
       \dd\Big(\sum_{m\ne n}a_m\beta^m_s\Big), \quad n \geq 1.
\]
Conditional on the entire $n$-th coordinate (i.e.\ for fixed $\omega_n$ with respect to $\P^{\not n}$), the family $(V_n(t))_{t\in I\cap\Q}$
 is a centred Gaussian and therefore symmetric. It follows that 
\begin{align*}
    \varepsilon_0 &<\P\Big(\sup_{t\in I\cap\Q}\norm{a_n U^{\beta^n}_{g_n}(t)}>n\Big) \\ 
    & \leq \P\Big(\sup_{t\in I\cap\Q}\norm{a_n U^{\beta^n}_{g_n}(t) + V_n(t)}>n\Big) + \P\Big(\sup_{t\in I\cap\Q}\norm{a_n U^{\beta^n}_{g_n}(t) - V_n(t)}>n\Big)
\\ 
    & = \E^n \P^{\not n}\Big(\sup_{t\in I\cap\Q}\norm{a_n U^{\beta^n}_{g_n}(t) + V_n(t)}>n\Big) + \E^n \P^{\not n}\Big(\sup_{t\in I\cap\Q}\norm{a_n U^{\beta^n}_{g_n}(t) - V_n(t)}>n\Big)  \\
    & = 2 \E^n \P^{\not n}\Big(\sup_{t\in I\cap\Q}\norm{a_n U^{\beta^n}_{g_n}(t) + V_n(t)}>n\Big)  
    \\ & = 2 \P\Big(\sup_{t\in I\cap\Q}\norm{U^\beta_{g_n}(t)}>n\Big),
\end{align*}
showing that~\ref{it2:extrap} does not hold on the product stochastic basis. 

\end{proof}

\begin{proof}[Proof of \ref{it3:extrap} implies \ref{it5:extrap}.]
Write $M = \sup_{t \in I}\norm{S(t)}_{\mathscr{L}(X)}$. Let $p>0$ and let $g$ be a predictable step process with values in $\cD(A)$.  Let $\theta > \max\{1,M\}$ and $0 < \varepsilon < \theta^{-p}$. Set $\delta =(\theta - M)/ (2K_{\varepsilon})$.
Taking into account that the constant $K_{\varepsilon}$ can be chosen independent of the stochastic basis, an adaptation of the three stopping time argument of~\cite[Lemma 3.5]{coxvw2024sharp} (which originates from~\cite{Burkholder1973}), results in the following good-$\lambda$ inequality:
\begin{align}\label{eq:goodlambda}
    \P(\norm{U_g}_{C(I;X)} > \theta \lambda, \norm{g}_{L^2(I;X)} \leq \delta \lambda) \leq \varepsilon \P(\norm{U_g}_{C(I;X)} \geq \lambda), \quad \lambda > 0.
\end{align}
For convenience of the reader, we briefly indicate the details. 

Fix $\lambda>0$. Let $\sigma$ be the first time that
$\norm{U_g}$ reaches $\lambda$, with $\sigma= \sup I$ if no such time exists. Set $V_{\lambda}= \{\sigma<\sup I\}$. 
Then
\[
    \{\norm{U_g}_{C(I;X)}>\theta\lambda\}\subseteq V_{\lambda}
    \subseteq \{\norm{U_g}_{C(I;X)}\geq\lambda\}.
\]
If $\P(V_{\lambda})=0$, \eqref{eq:goodlambda} is immediate.
Otherwise, define $\mathbb Q_\lambda=\P(\,\cdot\mid V_{\lambda})$.
Since $V_\lambda\in\F_\sigma$, the strong Markov property implies that
$\widehat\beta_s\coloneq\beta_{\sigma+s}-\beta_\sigma$
defines a Brownian motion under $\mathbb Q_\lambda$ with respect to the
shifted filtration $(\F_{\sigma+s})_{s\geq0}$.
Extend $g$ by zero beyond the original time interval and put
$\widehat g(s)=g(\sigma+s)$. Define
\[
    \rho=\inf\left\{
        r\geq0:
        \int_0^r\norm{\widehat g(s)}^2\,\dd s
        \geq\delta^2\lambda^2
    \right\},
    \qquad
    h=\widehat g\,\one_{[0,\rho]},
\]
where $\inf\varnothing=\infty$. Then $h$ is predictable for the
shifted filtration and
\[
    \norm{h}_{L^\infty(\mathbb Q_\lambda;L^2(I;X))}
    \leq\delta\lambda.
\]
Write
\[
    U_h(s)=\int_0^s S(s-r)h(r)\,\dd\widehat\beta_r,
    \qquad s\in I.
\]
Since $g$ is in $L^2(I;\cD(A))$ a.s., the same holds for $h$, and therefore, $U_h$ has a continuous version. By \ref{it3:extrap}, with the same constant $K_\varepsilon$ on this
new stochastic basis,
\[
    \mathbb Q_\lambda\left(
        \norm{U_h}_{C(I;X)}
        >K_\varepsilon\delta\lambda
    \right)\leq\varepsilon.
\]

On the event $\{\norm{U_g}_{C(I;X)}>\theta\lambda,\ \|g\|_{L^2(I;X)}\leq\delta\lambda\}$, one has $\rho = \sup I$, and hence
\[
    U_g(\sigma+s)=S(s)U_g(\sigma)+U_h(s),
    \qquad \sigma+s\in I.
\]
Moreover, continuity gives $\norm{U_g(\sigma)}=\lambda$.
Consequently, on this event,
\[
    \norm{U_h}_{C(I;X)}
    >(\theta-M)\lambda
    =2K_\varepsilon\delta\lambda.
\]
It follows that
\[
    \begin{aligned}
    \P(\norm{U_g}_{C(I;X)}>\theta\lambda,\ \|g\|_{L^2(I;X)}\leq\delta\lambda)
    &\leq \P(V_\lambda)\,
        \mathbb Q_\lambda\left(
            \norm{U_h}_{C(I;X)}
            >K_\varepsilon\delta\lambda
        \right)\\ &\leq\varepsilon\,\P(V_\lambda)\\ &\leq\varepsilon\,\P(\norm{U_g}_{C(I;X)}\geq\lambda),
    \end{aligned}
\]
which proves~\eqref{eq:goodlambda}.
Since $\theta > 1$, $\delta > 0$, $\varepsilon > 0$, and $\varepsilon \theta^p < 1$, we may now apply~\cite[Lemma 7.1]{Burkholder1973} using~\eqref{eq:goodlambda} to find
\begin{align}
    \E\norm{U_g}_{C(I;X)}^p \leq \frac{\theta^p \delta^{-p}}{1 - \theta^p \varepsilon} \E \norm{g}_{L^2(I;X)}^p,
\end{align}
hence $\cM_{p,I}(S) < \infty$.
\end{proof}

The interval $I$ was fixed throughout Proposition~\ref{prop:extrap}.
However, the next proposition shows that finiteness of $\cM_{p,I}(S)$ does not depend on $I$ if $S$ is exponentially stable.

\begin{proposition}[Independence of the interval]
    \label{prop:indepinterval}
    Suppose that $S$ is exponentially stable and let $I = [0,T]$ with $T > 0$. 
    Then we have $\cM_{2,I}(S) < \infty$ iff $\cM_{2,\R_+}(S) < \infty$, where we used constants which are uniform over all stochastic bases. 
\end{proposition}

\begin{proof}
We only prove the forward implication; the converse is trivial.
Assume $\cM_{2,I}(S) < \infty$ and let $g$ be a predictable step process.
For $n \geq 0$ and $r \in [0,T]$, we define
\begin{align}
    V_n(r)=\int_{nT}^{nT+r}S(nT+r-s)g(s)\,\dd\beta_s,
\end{align}
so that (when using continuous versions), the semigroup identity gives almost surely
\begin{align}
    U_g(nT + r) =\sum_{k < n} S(r + (n - k - 1)T)V_k(T) + V_n(r), \quad n \geq 0, \, r \in [0,T].
\end{align}
Using the exponential stability of $S$, we may use Cauchy--Schwarz to estimate
\begin{align}
    \sup_{t \geq 0}\norm{U_g(t)}^2 \lesssim \sum_{n\geq 0} \sup_{r \in [0,T]}\norm{V_n(r)}^2.
\end{align}
Taking expectations and using Fubini's theorem, we may
 apply the maximal estimate individually to each $V_n$ with the Brownian motion $t \mapsto \beta_{nT + t} - \beta_{nT}$ (taking into account the uniformity over the stochastic basis) and find
\begin{align}
    \E \sup_{t \geq 0}\norm{U_g(t)}^2 \lesssim \sum_{n}\E \norm{g}_{L^2([nT,(n+1)T];X)}^2 = \E\norm{g}_{L^2(\R_+;X)}^2,
\end{align}
so that indeed $\cM_{2,\R_+}(S) < \infty$.
\end{proof}

\begin{proof}[Proof of Theorem~\ref{thm:main}]
Take the operator $A$ from Proposition~\ref{prop:sqrfuncexample}.
Then it follows that \ref{it:main1} holds.
Moreover, since the lower square-function estimate does not hold, 
Theorem~\ref{thm:necessary-square} gives  that
$\cM_{2,\R_+}(S)=\infty$.
By Proposition~\ref{prop:indepinterval} we obtain $\cM_{2,I}(S)=\infty$ for $I = [0,1]$. Finally an application of Proposition~\ref{prop:extrap} shows that there exists a stochastic basis and $g\in \mathcal{E}_\infty(I)$ such that $\sup_{t \in I\cap \Q} \norm{U_g(t)} = \infty$ occurs with nonzero probability.  
\end{proof}

\begin{proof}[Proof of Theorem~\ref{thm:main2}]
Take the operator $A$ from Theorem~\ref{thm:endpoint-counterexample}.
Then \ref{it1:main2} holds. The two-sided square-function estimates and McIntosh's theorem,
recalled in Section~\ref{sec:squarefunc}, give
$\omega_{H^\infty}(A)=\omega(A)=\pi/2$.
Thus \ref{it2:main2} holds.
Since $\cM_{2,\R_+}(S)=\infty$, the remaining assertions follow
exactly as in the proof of Theorem~\ref{thm:main}.
\end{proof}

\begin{remark}\label{rem:fixed-basis}
Applying \cite[Lemma~1.3(b)]{Kurtz2014} successively to the coefficients of predictable step processes, followed by tedious approximation arguments, one can transfer the tail and maximal estimates between stochastic bases with unchanged constants. Thus Proposition~\ref{prop:extrap} and Theorems~\ref{thm:main} and~\ref{thm:main2} can also be formulated on any fixed stochastic basis.  \end{remark}

\begin{remark}\label{rem:extrapolationMp}
With a variation of the good-$\lambda$ argument in the implication \ref{it3:extrap}$\to$\ref{it5:extrap}, one can show that $\cM_{p_0,I}(S)<\infty$ implies $\cM_{p,I}(S)<\infty$ for all $p\in (0,\infty)$ on a fixed stochastic basis. Indeed, it suffices to show \eqref{eq:goodlambda}. 
Choose $\theta>\max\{1,M\}$, and take $\delta>0$ so small that
\[
    \varepsilon :=\left(\frac{2\cM_{p_0,I}(S) \delta}{\theta-M}\right)^{p_0}
    <\theta^{-p}.
\]
Let $g$ be a predictable step process. Fix $\lambda>0$. Let $\sigma$ be the first time that $\norm{U_g}$
reaches $\lambda$, and let $\tau$ be the first time that
$\int_0^t\norm{g(s)}^2\,\dd s$ reaches $4\delta^2\lambda^2$,
with the usual terminal-time conventions. Define $h(t)=\one_{\{\sigma<t\leq\tau\}}g(t)$. This process is predictable on the original stochastic basis, and
\[
    \norm{h}_{L^2(I;X)}
    \leq 2\delta\lambda\,\one_{\{\norm{U_g}_{C(I;X)}\geq\lambda\}}.
\]
The assumed maximal estimate, and Markov's inequality therefore give
\[
    \P\Big(\sup_{t\in I}\norm{U_h(t)}
        >(\theta-M)\lambda\Big)
    \leq
    \frac{\cM_{p_0,I}(S)^{p_0}\E\norm{h}_{L^2(I;X)}^{p_0}}
         {((\theta-M)\lambda)^{p_0}} \leq \varepsilon \P(\norm{U_g}_{C(I;X)}\geq\lambda).
\]
On $\{\norm{U_g}_{C(I;X)}>\theta\lambda,\ \norm{g}_{L^2(I;X)}\leq\delta\lambda\}$, one has $\tau = \sup I$, and thus, 
\[
    U_h(t)=U_g(t)-S(t-\sigma)U_g(\sigma), \ \ \  t\geq\sigma, 
\]
Since $\norm{U_g(\sigma)}=\lambda$, this implies $\sup_{t\in I}\norm{U_h(t)}>(\theta-M)\lambda$, and hence \eqref{eq:goodlambda} follows.
\end{remark}

\section{A maximal estimate for a smaller class of integrands}
\label{sec:special-processes}
We have seen that for general semigroups, maximal estimates with predictable integrands which are in the space $L^\infty(\Omega;L^2(I;X))$ may fail.
We will now show that a maximal estimate can be recovered upon restricting the integrands to the space 
\[
\{g:[0,T]\times \Omega\to X: g \ \text{is predictable and } \ \int_0^T \|g(t,\cdot)\|_{L^\infty(\Omega;X)}^2 \dd t<\infty\}.\]
Then this space is a strict subset of \(L^\infty(\Omega;L^2(I;X))\), and includes all deterministic integrands in $L^2(I;X)$.
We avoid the notation \(L^2(I;L^\infty(\Omega;X))\) to avoid measurability and approximation problems in Bochner spaces in the potentially nonseparable space $L^\infty(\Omega;X)$. 
As will become apparent from the proof, the maximal estimate holds because the smaller space of integrands allows for deterministic control of $g$ at any fixed time $t \in I$.

\begin{theorem}[Energy-adapted chaining]
\label{thm:chaining}
Let $I = [0,T]$ for $T > 0$.
For every predictable $g:[0,T]\times \Omega\to X$ such that
\[\int_0^T \|g(t,\cdot)\|_{L^\infty(\Omega;X)}^2 \dd t<\infty,\]
the process $U_g$ has a continuous version.
Moreover, for every $p \in [1,\infty)$ one has
\begin{equation}
 \left\|
  \sup_{t \in I}\norm{U_g(t)}
 \right\|_{L^p(\Omega)}
 \leq C_p M_I^2 \Big(\int_0^T \|g(s,\cdot)\|_{L^\infty(\Omega;X)}^2 \dd s\Big)^{1/2},
 \label{eq:chaining-estimate}
\end{equation}
where $C_p = \frac{2^{3/4}B_{\max\{p,4\}}} {2^{1/4} - 1}$, $M_I = \sup_{t \in I}\norm{S(t)}_{\mathscr{L}(X)}$, and $B_q$ is as in \eqref{eq:BDG}. 
\end{theorem}

\begin{proof}
By rescaling time, we only need to prove the case $T = 1$. We first prove  the result in the case that $g\in L^\infty([0,1]\times \Omega;X)$. 
Since $g$ is uniformly bounded, the continuity follows from the factorization method \cite{DaPratoKwapienZabczyk1987}. Therefore, it remains to prove the estimate \eqref{eq:chaining-estimate}.
By homogeneity  we may additionally assume that $\int_0^1 \|g(t,\cdot)\|_{L^\infty(\Omega;X)}^2 \dd t = 1$.

For an interval $(a,b] = J \subset [0,1]$, we introduce by analogy with $U_g(t)$ the notation
\begin{align*}
    U_g(J) \coloneq \int_a^b S(b-s)g(s) \dd \beta_s.
\end{align*}

We begin by partitioning the time interval in a way which is tailored to $g$.
We define $\theta \colon [0,1] \to [0,2]$ by
\begin{equation}
    \theta(t) =\int_0^t 1 + \norm{g(s,\cdot)}^2_{L^{\infty}(\Omega;X)} \dd s.
\end{equation}
Using dominated convergence, we see that $\theta$ is continuous, strictly increasing and we have $\theta(0) = 0$ and $\theta(1) = 2$.
Thus, $\theta$ has a continuous strictly increasing inverse $\varrho \colon [0,2] \to [0,1]$, from which we define the binary nested sequence of partitions
\begin{align*}
    J_{n,k} &\coloneq (\varrho(k 2^{-n+1}), \varrho((k+1)2^{-n+1})], \qquad 0 \leq k < 2^n, \\
    \mathcal{D}_n &\coloneq \{J_{n,k} \}_{0 \leq k < 2^n}.
\end{align*}
Finally, we set $\mathcal{T} \coloneq \{ \varrho(k 2^{-n+1}) : n \geq 0, \, 0 \leq k < 2^n \}$.
Note that $\mathcal{T}$ is dense in $[0,1]$, since $\varrho$ is continuous and surjective.

We record the fact that for any $J \in \mathcal{D}_n$, we have $\int_J \norm{g(t,\cdot)}_{L^{\infty}(\Omega;X)}^2 \dd t \leq 2^{-n+1}$ by construction.
Thus, for $q \in (0,\infty)$, the bound \eqref{eq:BDG} gives
\begin{equation}
    \label{eq:bdgchain}
    \norm{U_g(J)}_{L^q(\Omega;X)} \leq \sqrt{2} B_q 2^{-n/2}M_{I}, \quad J \in \mathcal{D}_n.
\end{equation}
Note that the intervals $J \in \mathcal{D}_n$ do not depend on $\omega$, which is crucial for this inequality.

Now fix $t \in \mathcal{T}$.
It follows from our construction that $(0,t]$ can be partitioned into a finite set of intervals $\mathcal{A} \subset \bigcup_{n \geq 0}\mathcal{D}_n$ which contains at most one interval from each $\mathcal{D}_n$.
Using the semigroup identity and letting $r_{J}$ denote the right endpoint of $J$, it then follows that almost surely
\[
 U_g(t)=\sum_{J\in \mathcal{A}}S(t-r_J)U_g(J).
\]
Since $t$ was arbitrary, this results in the almost sure inequality
\begin{equation}
\label{eq:treebound}
 \sup_{t \in \mathcal{T}}\norm{U_g(t)}
 \leq M_{I} \sum_{n \geq 0} \max_{J \in \mathcal{D}_n}\norm{U_g(J)}.
\end{equation}
We want to estimate the right-hand side term by term in $L^p(\Omega)$.
To do this, note that for $q \in [p,\infty)$, the embeddings $L^q(\Omega) \hookrightarrow L^p(\Omega)$ and $\ell^q \hookrightarrow \ell^{\infty}$ combined with Fubini's theorem give 
\begin{align}
    \label{eq:Dnbound}
    \norm{\max_{J \in \mathcal{D}_n}\norm{U_g(J)}}_{L^p(\Omega)}
    \leq 2^{n/q} \max_{J \in \mathcal{D}_n} \norm{U_g(J)}_{L^q(\Omega;X)} \leq \sqrt{2} B_q M_{I} 2^{n/q - n/2},
\end{align}
with the second step following from~\eqref{eq:bdgchain}.
Choosing $q = \max\{4,p\}$ the right-hand side of~\eqref{eq:Dnbound} becomes summable in $n$, so the triangle inequality applied to~\eqref{eq:treebound} results in
\begin{equation}
    \norm{\sup_{t \in \mathcal{T}} \norm{U_g(t)}}_{L^p(\Omega)} \leq 2^{1/2}M_{I}^2 B_{q} \sum_{n\geq 0}  2^{n/q - n/2} \leq \frac{2^{3/4}M_{I}^2 B_{q}} {2^{1/4} - 1}.
\end{equation}
By density, we may replace $\mathcal{T}$ by $[0,1]$ in the above estimate (using the continuous version of $U_g$), which concludes the proof of the case where $g$ is uniformly bounded.

To treat the general case we cannot use approximation by adapted step processes, since these are not dense in the norm we are using. Instead, we can use an approximation argument through truncations. Indeed, 
let $g_k(s) = g(s)\one_{\norm{g(s)} \leq k}$. Applying \eqref{eq:chaining-estimate} to the bounded processes $g_k - g_\ell$, it follows that $(U_{g_k})_{k\geq 1}$ is a Cauchy sequence in $L^p(\Omega;C([0,1];X))$ and hence convergent. Since the limit is a version of $U_g$, this completes the proof of the continuity for general $g$. The estimate \eqref{eq:chaining-estimate} follows by taking limits as well.  
\end{proof}

We also provide a different proof, which makes use of a time change to reduce to the case where $g \in L^{\infty}([0,1]\times \Omega ;X)$. Note that this proof might not recover the same constant in~\eqref{eq:chaining-estimate}.

\begin{proof}[Alternative proof]
By the same reductions as in the proof above, it is enough to prove~\eqref{eq:chaining-estimate} for $T = 1$ and $g\in L^\infty([0,1]\times \Omega;X)$ with the normalization $\int_0^1 \|g(t,\cdot)\|_{L^\infty(\Omega;X)}^2 \dd t = 1$ (in which case $U_g$ has a continuous version).
We let $\varrho$ be as in the proof above, and we will make a time change according to $t = \varrho(t')$.
To do this, we define the following objects:
\begin{align*}
    \widetilde{S}(t',s') &= S(\varrho(t') - \varrho(s')),  \qquad 0 \leq s' \leq t' \leq 2,\\
    \widetilde{\beta}(t') &= \int_0^{\varrho(t')}\sqrt{1 + \norm{g(s,\cdot)}_{L^{\infty}(\Omega;X)}^2} \dd \beta_s, \qquad t' \in [0,2], \\
    \widetilde{g}(s') &= \frac{g(\varrho(s'))}{\sqrt{1 + \norm{g(\varrho(s'),\cdot)}_{L^{\infty}(\Omega;X)}^2}}, \qquad s' \in [0,2].
\end{align*}
Note that $\widetilde{S}$ is a $C_0$-evolution family (by continuity of $\varrho$), $\widetilde{\beta}$ is a Brownian motion with respect to a time-changed filtration (by L\'evy's characterization), and $\tilde{g}$ is predictable with respect to this filtration.
By the deterministic time change formula, we then obtain the identity
\begin{align*}
    U_g(\varrho(t')) = \int_0^{t'}\widetilde{S}(t',s')\widetilde{g}(s') \dd \widetilde{\beta}_{s'}, \qquad t' \in [0,2].
\end{align*}
After observing that $\norm{\tilde{g}(s')} \leq 1$ for all $s' \in [0,2]$, we may apply the factorisation method \cite{Seidler1993} to the right-hand side to find
\begin{align}
    \left\lVert \sup_{t' \in [0,2]} \norm{U_g(\varrho(t'))}_{X}\right\rVert_{L^p(\Omega)} \lesssim M_I^2 \sqrt{p}.
\end{align}
Since $\varrho$ is surjective onto $[0,1]$, this exactly coincides with~\eqref{eq:chaining-estimate}. \qedhere

\end{proof}

\begin{remark}[Martingale type~\(2\)]
The Hilbert space assumption in Theorem~\ref{thm:chaining} can be
relaxed to assuming that \(X\) is a separable Banach space of
martingale type~\(2\). 
The only change in the proof is that one needs to use the martingale type-2 version of the Burkholder--Davis--Gundy inequality instead of the Hilbert space version (see \cite{Sei10}).
The constant $C_p$ in~\eqref{eq:chaining-estimate} will subsequently incur a dependence on the geometry of $X$. 

Moreover, the one-dimensional Brownian motion may be replaced by a cylindrical Brownian motion on a Hilbert space $H$. The integrability condition on $g$ needs to be replaced by 
\[\int_0^T \|g(t,\cdot)\|_{L^\infty(\Omega;\gamma(H,X))}^2 \dd t<\infty.\]

Here $\gamma(H,X)$ is as in \cite[Chapter 9]{HNVW2017}, and coincides with the Hilbert--Schmidt operators when $X$ is a Hilbert space. 
\end{remark}

\begin{remark}[Evolution families]
The semigroup may also be replaced by a (deterministic) strongly continuous evolution family
\[
 \mathscr S=(\mathscr S(t,s))_{0\leq s\leq t\leq T}
 \subseteq\mathcal L(X),
\]
with an analogous proof.
\end{remark}

\subsection*{AI disclosure statement}
GPT 5.6 and 6 Pro were used during the preparation of this paper to explore proof strategies, organize intermediate {\LaTeX} drafts, and to assist in checking correctness. The authors retain full responsible for the content of the paper.

\bibliographystyle{alpha}
\bibliography{smr}

\end{document}